\documentclass[12pt]{amsart}
\usepackage[margin=1in]{geometry}
\usepackage[hidelinks]{hyperref}

\usepackage{amsmath,amsthm,amssymb}

\usepackage{tikz}

\usepackage{booktabs}  
\usepackage{mathtools} 
\usepackage{units}     

\theoremstyle{plain}
\newtheorem{theorem}{Theorem}
\newtheorem{proposition}{Proposition}[section]
\newtheorem{corollary}[theorem]{Corollary}
\newtheorem{lemma}[proposition]{Lemma}
\newtheorem{innercustomthm}{Theorem}
\newenvironment{customthm}[1]
  {\renewcommand\theinnercustomthm{#1}\innercustomthm}
  {\endinnercustomthm}
\newtheorem{innercustomcor}{Corollary}
\newenvironment{customcor}[1]
  {\renewcommand\theinnercustomcor{#1}\innercustomcor}
  {\endinnercustomcor}

\theoremstyle{definition}
\newtheorem{definition}[proposition]{Definition}
\newcommand{\defn}[1]{\textbf{#1}}

\theoremstyle{remark}
\newtheorem*{example}{Example}

\newcommand{\script}[1]{\mathcal{#1}}              
\newcommand{\defeq}{\coloneqq}                     
\newcommand{\set}[2]{\left\{{#1}\,:\,{#2}\right\}} 
\newcommand{\card}[1]{\#\!\left(#1\right)}         
\newcommand{\union}{\cup}                          
\newcommand{\intersect}{\cap}                      
\newcommand{\NN}{\mathbb{N}}                       
\newcommand{\ZZ}{\mathbb{Z}}                       
\newcommand{\QQ}{\mathbb{Q}}                       
\newcommand{\RR}{\mathbb{R}}                       
\newcommand{\abs}[1]{\left|#1\right|}              
\newcommand{\vect}[1]{\Hat{#1}}                    
\DeclareMathOperator{\lex}{lex}                    

\newcommand{\Semi}{\script{S}}                     
\newcommand{\Poly}{\script{P}}                     
\newcommand{\Cone}{\script{C}}                     

\title[Quasi-polynomial growth of numerical and affine semigroups]{Quasi-polynomial growth of numerical and \\ affine semigroups with constrained gaps}

\author{Michael DiPasquale}
\address{Dept.~of Mathematics and Statistics, University of South Alabama, Mobile, AL, USA}
\email{mdipasquale@southalabama.edu}

\author{Bryan R.~Gillespie \and Chris Peterson}
\address{Dept.~of Mathematics, Colorado State University, Fort Collins, CO, USA}
\email{bryan.gillespie@colostate.edu}
\email{peterson@math.colostate.edu}


\begin{document}

\begin{abstract}
    A common tool in the theory of numerical semigroups is to interpret a desired class of semigroups as the integer lattice points in a rational polyhedron in order to leverage computational and enumerative techniques from polyhedral geometry.  Most arguments of this type make use of a parametrization of numerical semigroups with fixed multiplicity $m$ in terms of their $m$-Ap\'{e}ry sets, giving a representation called Kunz coordinates which obey a collection of inequalities defining the Kunz polyhedron.  In this work, we introduce a new class of polyhedra describing numerical semigroups in terms of a truncated addition table of their sporadic elements.  Applying a classical theorem of Ehrhart to slices of these polyhedra, we prove that the number of numerical semigroups with $n$ sporadic elements and Frobenius number $f$ is polynomial up to periodicity, or quasi-polynomial, as a function of $f$ for fixed $n$.  We also generalize this approach to higher dimensions to demonstrate quasi-polynomial growth of the number of affine semigroups with a fixed number of elements, and all gaps, contained in an integer dilation of a fixed polytope.
\end{abstract}

\maketitle

\section{Introduction}

A \emph{numerical semigroup} is a cofinite additive submonoid of $\NN = \{0, 1, 2, \ldots\}$, i.e.\ a subset of the nonnegative integers with finite complement which contains 0 and is closed under taking sums.  The discrete, number-theoretic nature of such semigroups lends them to interesting combinatorial analysis, and in recent decades a significant interest has grown in enumerative and algorithmic questions: given one or more integer statistics assigned to numerical semigroups, how many numerical semigroups are there with specified values for these statistics, and how can the set of such semigroups be efficiently represented and computed?  If one statistic is allowed to vary, what is the asymptotic behavior of these sets?

Several important statistics feature prominently in this enumerative study.  We write $\Semi$ for the collection of all numerical semigroups.  If $S \in \Semi$, then the elements of $\NN \setminus S$ are called the \emph{gaps} or \emph{holes} of $S$, and the number of gaps is called the \emph{genus} of $S$.  The largest gap of $S$ is called its \emph{Frobenius number}, and the smallest positive element of $S$ is called its \emph{multiplicity}.  The elements of $S$ less than its Frobenius number are called \emph{sporadic elements}.  Any numerical semigroup contains a distinguished finite collection of \emph{minimal generators} of size at most its multiplicity, given by the positive semigroup elements not expressible as a sum of two other positive elements.  The number of minimal generators of $S$ is called its \emph{embedding dimension}, and a semigroup whose embedding dimension is equal to its multiplicity is said to have \emph{maximal embedding dimension}.

A fundamental construction which has used to approach such questions is the \emph{semigroup tree}, a rooted tree structure imposed on $\Semi$ by starting with root $S_0 = \NN$, and letting $S'$ be a descendent of $S$ if $S = S' \union \{f\}$ for $f$ the Frobenius number of $S'$.  It is easy to see that the levels of the semigroup tree enumerate the numerical semigroups by increasing genus, so the construction plays a deep role in understanding the function $N(g)$ counting the total number of numerical semigroups with genus $g$.  In 2008, Bras-Amor\'os~\cite{bras2008fibonacci} conjectured based on computational evidence that the number of such semigroups grew at a rate approximating a Fibonacci recurrence, implying a limiting ratio on the sizes of successive layers of the semigroup tree given by the golden ratio $\varphi = (1 + \sqrt{5}) / 2$.  Subsequent analysis of the semigroup tree in~\cite{bras2009bounds} and~\cite{elizalde_improved_2010} produced bounds on the growth of $N(g)$ which supported the conjectured asymptotics.  Then in 2013, building on work of Bras-Amar\'os and Bulygin~\cite{bras-amoros_towards_2009} and Zhao~\cite{zhao_constructing_2010}, Alex Zhai proved in~\cite{zhai2013fibonacci} that $\lim_{g\rightarrow \infty} N(g) / \varphi^g = C$ for some constant $C \geq 3.78$, confirming the conjectured exponential rate of growth.  A survey by Kaplan~\cite{kaplan2017counting} provides a lucid exposition on the background and proof techniques of this result.  Further results related to the semigroup tree can be found in~\cite{blanco2012enumeration,blanco2013tree,branco2021set,rosales-1996}.

Another perspective which has proven useful for counting arguments and computations is to represent the numerical semigroups of a desired class as the integer lattice points contained in a rational polyhedron, which allows theoretical and algorithmic tools from convex geometry to be applied.  The most studied construction of this type is the \emph{Kunz polyhedron}, which was introduced in~\cite{kunz-1987-klassifikation}, and independently in~\cite{rosales_systems_2002}.  The integer lattice points of the $m$-Kunz polyhedron in $\RR^{m-1}$ represent the numerical semigroups with multiplicity $m$ by their \emph{Kunz coordinates}, defined for a semigroup $S$ as the unique tuple $(k_1, \ldots, k_{m-1})$ such that $k_i$ is minimal satisfying $k_i m + i \in S$ for each $i$.

One way that the Kunz polyhedron and Kunz coordinates have been used is as a concrete geometric setting in which computing the semigroups of a desired type is faster and simpler than with a less structured representation; see for instance~\cite{blanco2011counting,blanco2012enumeration,blanco_set_2012,blanco2013tree,branco2021set}.  Another application has been to apply Ehrhart theory, a class of results characterizing the number of integer lattice points in families of rational polytopes, to enumerative questions about numerical semigroups.  In~\cite{kaplan2012counting}, this approach was used to show that the numerical semigroups of fixed multiplicity $m$ and genus $g$ are \emph{eventually quasi-polynomial} as a function of $g$, i.e.\ that there is a modulus $N$ and polynomial functions $p_i, i = 0, \ldots, N-1$ such that for large enough $g$ the number of such semigroups is given by $p_i(g)$ for $g \equiv i \pmod{N}$.  Later in~\cite{alhajjar2019numerical}, the quasi-polynomial characterization was applied to show that asymptotically in $g$, almost all numerical semigroups of genus $g$ and fixed multiplicity $m$ have maximum embedding dimension.

In this work, we introduce a new polyhedral representation of numerical semigroups described in terms of the addition relations between their sporadic elements.  Specifically, sums of the positive sporadic elements $x_1 < \cdots < x_n$ of a numerical semigroup $S$ with Frobenius number $f$ can be described by a function $\tau : \{1, \ldots, n\}^2 \to \{1, \ldots, n, \infty\}$, where $\tau(i, j) = k < \infty$ if $x_i + x_j = x_k$, and $\tau(i, j) = \infty$ if $x_i + x_j > f$.  We call such a function a \emph{truncated addition table}, or just an \emph{addition table}, over $n$ elements.  Given an addition table $\tau$, we define a polytope $\Poly_\tau$ called the \emph{sporadic relation polytope}, or \emph{SR-polytope}, such that the integer lattice points of $f \Poly_\tau$ correspond with the numerical semigroups with Frobenius number $f$ whose sporadic elements have addition table $\tau$:

\begin{customthm}{\ref{thm:num-semigroup-sr-polytope-bijection}}
	Let $n \in \NN$, and let $\tau$ be an addition table over $n$ elements.  Then for positive $f \in \NN$, the integer lattice points in $f \Poly_\tau$ are in one-to-one correspondence with the numerical semigroups with Frobenius number $f$ whose positive sporadic elements have addition table $\tau$ via the correspondence
	\[
	(x_1, \ldots, x_n) \mapsto \{0\} \union \{x_1, \ldots, x_n\} \union \{f+1, \rightarrow \}
	\]
\end{customthm}

A classical result in Ehrhart theory states that the number of integer lattice points in the dilations $fP$ of a rational polytope $P$ is given by a quasi-polynomial, so by summing over the possible addition tables $\tau$, we are able to conclude the following enumerative result.

\begin{customcor}{\ref{cor:num-semigroups-quasipoly}}
	For fixed $n \in \NN$, the number of numerical semigroups with $n+1$ sporadic elements and Frobenius number $f$ is a quasi-polynomial function of $f$ with degree $n$ and constant leading coefficient $1/(2^n n!)$.
\end{customcor}

The principle of describing a discrete semigroup in terms of an appropriate truncated addition table is general enough that it may be applied in more complicated settings.  As an example of this, we present a related construction for multi-dimensional affine semigroups.  A \emph{$d$-dimensional $\Cone$-semigroup} with respect to a pointed rational cone $\Cone \subseteq \RR^d$ is a cofinite additive submonoid of $\Cone \intersect \ZZ^d$, or equivalently, a subset of $\Cone \intersect \ZZ^d$ which has finite complement in $\Cone$, contains 0, and is closed under addition.  We say that a rational polytope $P$ containing a neighborhood of the origin in $\Cone$ is $\Cone$-\emph{compatible} if for any $x \in \Cone \setminus P$, the translated cone $x + \Cone$ is a subset of $\Cone \setminus P$.  For an addition table $\tau$ over $n$ elements and a $\Cone$-compatible polytope $P$, we define a set $\Poly_\tau^{(\Cone,P)} \subseteq \RR^{d \times n}$ similar to the polytope $P_\tau$ defined for numerical semigroups which is instead a union of a (large) number of polytopes with mixed open and closed faces.  We prove an analogue of Theorem \ref{thm:num-semigroup-sr-polytope-bijection}, and establish the following enumerative result for affine semigroups:

\begin{customcor}{\ref{cor:affine-quasipolynomial-growth}}
	Let $\Cone \subseteq \RR^d$ be a pointed rational polyhedral cone, and let $P \subseteq \RR^d$ be a $\Cone$-compatible polytope.  Then for fixed $n \in \NN$, the number of $\Cone$-semigroups with $n+1$ elements in the polytope $\alpha P$ is a quasi-polynomial function of $\alpha$ with degree $nd$.
\end{customcor}

The remainder of the article will be organized as follows.  In Section \ref{sec:two-examples} we present two instructional examples motivating the general constructions that follow.  In Section \ref{sec:polytope-decomposition} we define the polytopes $P_\tau$ and prove Theorem \ref{thm:num-semigroup-sr-polytope-bijection} and Corollary \ref{cor:num-semigroups-quasipoly}, and in Section \ref{sec:affine-setting} we generalize these results to the setting of affine semigroups and prove Theorem \ref{thm:BijectionAllDim} and Corollary \ref{cor:affine-quasipolynomial-growth}.

\section{The first two cases} \label{sec:two-examples}

We will start by working out a few small examples by hand to illustrate the main idea.  Throughout the following two sections, we will write $\Semi(n, f)$ for the set of numerical semigroups with Frobenius number $f$ and $n$ positive sporadic elements, and $N(n, f)$ for the cardinality of this set.  When writing explicit numerical semigroups, we will use a right arrow ($\rightarrow$) to mean the set of numbers larger than the last number in some set brackets; for instance, $\{3, 5, 6, 8, \rightarrow\}$ means the set $\{3, 5, 6\} \union \{8, 9, 10, \ldots\}$.  We will assume basic familiarity with, but no specialized knowledge about, the topic of convex polytopes; see \cite{ziegler_lectures_1995} for a generous introduction.  For the remainder of this work we will write $[n]$ to denote the set $\{1, \ldots, n\}$.

Let us begin by enumerating the semigroups in $\Semi(1,f)$.  If $S \in \Semi(1,f)$, then $S \intersect [f]$ consists of a single element, which we denote by $x_1$.  Observe that $S$ is determined by $x_1$ and $f$, since $S = \{0, x_1, f+1, \rightarrow\}$.  Thus to count how many semigroups are in $\Semi(1,f)$ it suffices to count the possibilities for $x_1$.  (This observation also carries over to $\Semi(n,f)$ without any complications.)  So what properties must $x_1$ satisfy?  Since $S \intersect [f] = \{x_1\}$ and $x_1 + x_1 \in S$, we must have $2x_1 > f$, for otherwise, $2x_1$ would be a second element of $S$ in the interval $[f]$.  To ensure the restriction on the Frobenius number, we must also have $x_1 < f$.  As long as $x_1$ satisfies these properties, $S = \{0, x_1, f+1, \rightarrow\}$ is a semigroup in $\Semi(1,f)$.

It follows that every semigroup in $\Semi(1,f)$ corresponds uniquely to an integer $x_1 \in (f/2, f)$.  Notice that $(f/2, f)$ is the $f$-dilate of the interval $(1/2, 1)$ and we can determine that
\[
N(1,f) =
\begin{cases}
	f / 2 - 1/2 & f \text{ is odd} \\
	f / 2 - 1  & f \text{ is even}
\end{cases}
\]

Recall that a function $q : \ZZ \to \ZZ$ is called a \defn{quasi-polynomial} if there is a positive integer modulus $K$ such that the restriction of $q$ to congruence classes modulo $K$ is polynomial.  The degree of a quasi-polynomial is the largest degree that appears among its associated polynomials.  The function $N(1,f)$ is a degree 1 quasi-polynomial with period $K = 2$.

Now consider the somewhat more involved case of numerical semigroups in $\Semi(2, f)$.  A semigroup $S \in \Semi(2, f)$ must contain exactly two elements strictly between 0 and $f$, say $0 < x_1 < x_2 < f$, and must include all numbers larger than $f$.

What are the possibilities for $x_1$ and $x_2$?  Unlike the situation for $n=1$, there are multiple cases to consider.  Since $S$ is a semigroup, it must be the case that $x_1 + x_1$ is also an element of $S$.  If $x_1 + x_1 > f$, then $x_2$ may be chosen arbitrarily between $x_1$ and $f$, since in this case we also are assured that $x_1 + x_2$ and $x_2 + x_2$ are larger than $f$.  With these restrictions, we see that $S = \{0, x_1, x_2, f+1, \rightarrow\}$ works perfectly well.

The semigroups of this type thus correspond with the integer lattice points $(x_1, x_2)$ which satisfy the inequalities $0 < x_1 < x_2 < f$ and $x_1 + x_1 > f$.  The region described by these inequalities is the interior of the triangle with edges along the lines $x_1 = f/2$, $x_2 = f$, and $x_1 = x_2$.  Notice in particular that this is the $f$-dilate of a single fixed triangular region defined by the inequalities $0 < x_1 < x_2 < 1$ and $x_1 > 1/2$.  The number of lattice points in question can be computed as the number of pairs of distinct integers in the open interval $(f/2, f)$, so the desired semigroups are counted by
\[
q_1(f) = \begin{cases}
	\binom{(f-2)/2}{2} & f \text{ is even} \\
	\binom{(f-1)/2}{2} & f \text{ is odd}
\end{cases}
\]

We have one further case to consider, when $x_1 + x_1 < f$.  In this case, we must have $x_1 + x_1 = x_2$ since $x_1$ and $x_2$ are the only positive semigroup elements less than $f$.  This also imposes restrictions on the remaining sums, namely, that $x_1 + x_2$ and $x_2 + x_2$ are semigroup elements strictly larger than $x_2$, and so must also be larger than $f$.

Thus the numerical semigroups of this type correspond to lattice points $(x_1, x_2)$ satisfying $0 < x_1 < x_2 < f$, $x_1 + x_1 = x_2$, and $x_1 + x_2 > f$.  These are likewise lattice points in the $f$-dilate of a polytope, now the segment defined by $0 < x_1 < x_2 < 1$, $x_1 + x_1 = x_2$, and $x_1 + x_2 > 1$.  Cancelling $x_2$ with the equality $x_2 = 2x_1$, the defining inequalities for the region become $x_1 + x_2 = 3x_1 > f$ and $x_1 + x_1 = 2x_1 < f$.  Thus it suffices to count the number of integers $x_1$ in the interval $(f/3, f/2)$, each of which determines the corresponding value for $x_2$.  This number is given by a quasi-polynomial with period $K = 6$:
\[
q_2(f) = \begin{cases}
	(f-6) / 6, & f \equiv 0 \pmod{6} \\ 
	(f-1) / 6, & f \equiv 1 \pmod{6} \\ 
	(f-2) / 6, & f \equiv 2 \pmod{6} \\ 
	(f-3) / 6, & f \equiv 3 \pmod{6} \\ 
	(f-4) / 6, & f \equiv 4 \pmod{6} \\ 
	(f+1) / 6, & f \equiv 5 \pmod{6} \\ 
\end{cases}
\]
Combining the contributions of these two cases, we compute $N(2,f) = q_1(f) + q_2(f)$, or
\[
N(2,f)
= \begin{cases}
	f^2 / 8 + (-14f     ) / 24, & f \equiv 0 \pmod 6 \\
	f^2 / 8 + ( -8f +  5) / 24, & f \equiv 1 \pmod 6 \\
	f^2 / 8 + (-14f + 16) / 24, & f \equiv 2 \pmod 6 \\
	f^2 / 8 + ( -8f -  3) / 24, & f \equiv 3 \pmod 6 \\
	f^2 / 8 + (-14f +  8) / 24, & f \equiv 4 \pmod 6 \\
	f^2 / 8 + ( -8f + 13) / 24, & f \equiv 5 \pmod 6
\end{cases}
\]

We take two key observations away from the preceding computations.  First, we note that each semigroup in $\Semi(n,f)$, $f = 1, 2$, corresponds to a lattice point in the (relative) interior of the $f$-dilate of a polytope which may depend on algebraic relations between its sporadic elements.  Figure \ref{fig:rel-polytopes-n-2} shows the two polytopes computed above for the case of $n=2$.  Second, we observe that the number of such numerical semigroups is given by a quasi-polynomial of degree $n$ whose highest order term is independent of $f$.  We will see in Section \ref{sec:polytope-decomposition} that these two trends extend essentially as stated to the semigroups in $\Semi(n,f)$ for arbitrary $n$.

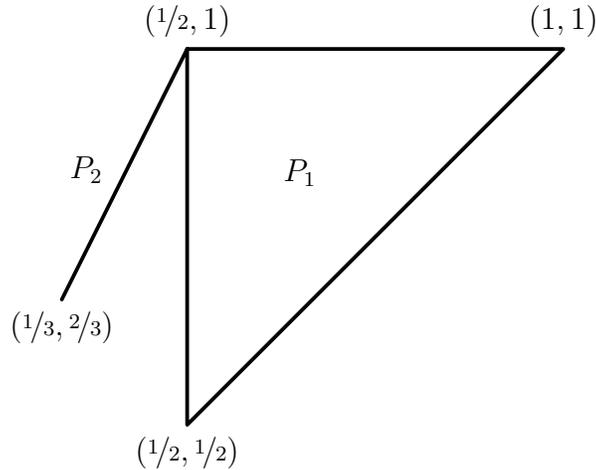
\begin{figure}[h]
	\begin{center}
		\begin{tikzpicture}[x=10cm,y=10cm]
			\draw[line width=0.5mm,line join=round,line cap=round] (1/3, 2/3)
			node[anchor=north] {$(\nicefrac{1}{3}, \nicefrac{2}{3})$}
			-- (1/2, 1)
			node[pos=0.42,label={[label distance=-2mm]120:$P_2$}] {}
			node[pos=1.0,anchor=south] {$(\nicefrac{1}{2}, 1)$}
			-- (1, 1)
			node[anchor=south] {$(1, 1)$}
			-- (1/2, 1/2)
			node[anchor=north] {$(\nicefrac{1}{2}, \nicefrac{1}{2})$}
			-- (1/2, 1);
			\node at (0.65, 0.835) {$P_1$};
		\end{tikzpicture}
	\end{center}
	\caption{Rational polytopes describing the numerical semigroups of $\Semi(2, f)$.  The integer lattice points contained in the relative interiors of the $f$-dilates of these polytopes correspond with the numerical semigroups with three sporadic elements and Frobenius number $f$.}
	\label{fig:rel-polytopes-n-2}
\end{figure}

\section{Enumeration of numerical semigroups and quasi-polynomial growth} \label{sec:polytope-decomposition}

We now present the general arguments suggested by the examples of Section \ref{sec:two-examples}.

\begin{definition}
	Let $n \in \NN$.  A \defn{truncated addition table}, or just \defn{addition table}, over $n$ elements is a function $\tau: [n]^2 \to [n] \union \{\infty\}$.  If $S$ is a numerical semigroup with Frobenius number $f$ and positive sporadic elements $x_1 < \cdots < x_n < f$, then the \defn{sporadic addition table} of $S$ is the truncated addition table $\tau$ over $n$ elements defined by
	\[
	\tau(i, j) = \begin{cases}
		k, & x_i + x_j = x_k \\
		\infty, & x_i + x_j > f
	\end{cases}
	\]
\end{definition}

For each truncated addition table $\tau$, we define a rational polytope encoding the algebraic relations specified by the table.  The points we will be interested in are the integer lattice points in the relative interior of this polytope.

\begin{definition}
	Let $n \in \NN$, and let $\tau$ be a truncated addition table over $n$ elements.  The \defn{sporadic relation polytope} or \defn{SR-polytope} of $\tau$ is the open polytope $\Poly_\tau \subseteq \RR^n$ defined by:
	\begin{itemize}
		\item \textit{Order inequalities}: $0 < x_1 < \cdots < x_n < 1$
		\item \textit{Sporadic relations}: $x_i + x_j = x_k$ for $\tau(i, j) = k < \infty$
		\item \textit{Truncation inequalities}: $x_i + x_j > 1$ for $\tau(i, j) = \infty$
	\end{itemize}
\end{definition}

Note that in the above, the SR-polytope is \emph{relatively open}, that is, it is open in the induced topology of the smallest affine subspace of $\RR^n$ which contains it.  Our first main result is that the polytope $\Poly_\tau$ represents the numerical semigroups with sporadic addition table $\tau$ in the following manner.

\begin{theorem}
	\label{thm:num-semigroup-sr-polytope-bijection}
	Let $n \in \NN$, and let $\tau$ be an addition table over $n$ elements.  Then for positive $f \in \NN$, the integer lattice points in $f \Poly_\tau$ are in one-to-one correspondence with the numerical semigroups with Frobenius number $f$ whose positive sporadic elements have addition table $\tau$ via the correspondence
	\[
	(x_1, \ldots, x_n) \mapsto \{0\} \union \{x_1, \ldots, x_n\} \union \{f+1, \rightarrow \}
	\]
\end{theorem}

\begin{proof}
	Note that the dilate $f \Poly_\tau$ is defined by the restrictions $0 < x_1 < \cdots < x_n < f$, $x_i + x_j = x_k$ for $\tau(i, j) = k < \infty$, and $x_i + x_j > f$ for $\tau(i, j) = \infty$.
	
	Suppose first that $(x_1, \ldots, x_n) \in f \Poly_\tau \intersect \ZZ^n$, and let $S = \{0\} \union \{x_1, \ldots, x_n\} \union \{f+1, \rightarrow\}$.  It is clear that $S$ contains 0 and has finitely many gaps, so it is only necessary to show that $S$ is closed under taking sums.  This is obvious for the sum of $0$ with any element and for the sum of any element larger than $f$ with any element, so we only need to argue that $x_i + x_j \in S$ for each $i, j$.  If $\tau(i, j) = k < \infty$, then the sporadic relations in $f \Poly_\tau$ give that $x_i + x_j = x_k \in S$.  Likewise, if $\tau(i, j) = \infty$, then the truncation inequalities imply that $x_i + x_j > f$, and since $x_i$ and $x_j$ are integers we see that their sum must be at least $f+1$, and so is in $S$.
	
	Now suppose that $S \in \Semi(n, f)$ with sporadic addition table $\tau$, and let $x_1 < x_2 < \cdots < x_n$ be the positive sporadic elements of $S$ in increasing order.  Then $S$ can be uniquely written as $\{0\} \union \{x_1, \ldots, x_n\} \union \{f + 1, \rightarrow\}$, and the integer lattice point $(x_1, \ldots, x_n)$ clearly satisfies the order inequalities for $f \Poly_\tau$.  The fact that it satisfies the sporadic relations and the truncation inequalities follows from the definition of the sporadic addition table of $S$.
	
	The fact that this map is invertible follows immediately from the order inequalities.
\end{proof}

Note that in general, the polytope $P_\tau$ may be empty if $\tau$ specifies algebraic relations which are not compatible with any numerical semigroup.  At a minimum, in order to represent a numerical semigroup as in Theorem \ref{thm:num-semigroup-sr-polytope-bijection}, the function $\tau$ must be symmetric and strictly increasing in both coordinates (but allowing for ``$\infty < \infty$'').  However, these necessary conditions are not also sufficient, and fully characterizing the addition tables representing the sporadic elements of some numerical semigroup seems to be a nuanced problem.

\begin{example}
	For the case of $n = 3$, there are a priori $4^6$ distinct symmetric truncated addition tables but only 4 of them are the sporadic addition tables of numerical semigroups, producing nonempty SR-polytopes.  These addition tables and their corresponding SR-polytopes are given in Table \ref{tbl:sg-polytopes-example}.  The addition tables are expressed by giving a list of their finite evaluations, up to reversing the order of the inputs.
	
	\begin{table}[h]
		\begin{center}
			\begin{tabular}{p{5cm} p{7cm}}
				\toprule
				Finite addition table values & Description of SR-polytope $\Poly_\tau$ \\
				\midrule
				--- &
				$0 < x_1 < x_2 < x_3 < 1$, \newline
				$x_1 + x_1 > 1$ \\
				\midrule
				$\tau(1, 1) = 2$ &
				$0 < x_1 < x_2 < x_3 < 1$, \newline
				$x_1 + x_1 = x_2$, \newline
				$x_1 + x_2 > 1$  \\
				\midrule
				$\tau(1, 1) = 3$ &
				$0 < x_1 < x_2 < x_3 < 1$, \newline
				$x_1 + x_1 = x_3$, \newline
				$x_1 + x_2 > 1$ \\
				\midrule
				$\tau(1, 1) = 2$, \newline
				$\tau(1, 2) = 3$ &
				$0 < x_1 < x_2 < x_3 < 1$, \newline
				$x_1 + x_1 = x_2$, \newline
				$x_1 + x_2 = x_3$, \newline
				$x_1 + x_3 > 1$, \newline
				$x_2 + x_2 > 1$ \\
				\bottomrule
			\end{tabular}
		\end{center}
		
		\caption{The nonempty SR-polytopes for $n = 3$}
		\label{tbl:sg-polytopes-example}
	\end{table}
\end{example}

We now turn our attention to the function $N(n, f)$ counting the number of numerical semigroups with $n$ positive sporadic elements and Frobenius number $f$.  Using the preceding characterization in terms of polytopes, we are able to apply the following fundamental result, originally due to Ehrhart in \cite{Ehrhart-1962,Ehrhart-1967} and MacDonald in \cite{macdonald_polynomials_1971}.
\begin{proposition}[\cite{beck_ehrhart-macdonald_2005}, Corollary 1.2]
	\label{prop:ehrhart-qp}
	Let $P \subseteq \RR^n$ be a rational convex polytope of dimension $d$.  Then the function
	\[
	L_P(m) \defeq \card{m P \intersect \ZZ^n}
	\]
	is a quasi-polynomial function of degree $d$.  The evaluation of this quasi-polynomial at negative integers gives
	\[
	L_P(-m) = (-1)^d L_{P^{\circ}}(m)
	\]
	where $P^\circ$ denotes the relative interior of $P$.
\end{proposition}

We can now prove the following.

\begin{corollary}
	\label{cor:num-semigroups-quasipoly}
	For fixed $n \in \NN$, the number of numerical semigroups with $n+1$ sporadic elements and Frobenius number $f$ is a quasi-polynomial function of $f$ with degree $n$ and constant leading coefficient $1/(2^n n!)$.
\end{corollary}

\begin{proof}
	For an addition table $\tau$, we write $\overline{\Poly}_\tau$ for the \emph{closed SR-polytope} of $\tau$, given by the same collection of restrictions as the open SR-polytope, but with all strict inequalities replaced by non-strict inequalities.  It can be shown that if $\Poly_\tau$ is nonempty, then it is equal to the relative interior of $\overline{\Poly}_\tau$.  Note that by Theorem \ref{thm:num-semigroup-sr-polytope-bijection}, this is the case exactly when there exists a numerical semigroup with sporadic addition table $\tau$.
	
	
	Let $n \in \NN$ be fixed, and let $\script{T}_n$ be the collection of addition tables over $n$ elements for which $\Poly_\tau$ is nonempty.  If $\Poly_\tau$ has dimension $d_{\tau} \geq 0$, then so does $\overline{\Poly}_\tau$, and so the lattice point counting function $L_{\overline{\Poly}_\tau}(f) = \card{f \overline{\Poly}_\tau \intersect \ZZ^n}$ is a quasi-polynomial $q_\tau$ of degree $d_\tau$ by Proposition \ref{prop:ehrhart-qp}.  Also by this proposition, we have
	\[
	L_{\overline{\Poly}_\tau}(-f) = q_\tau(-f) = (-1)^{d_\tau} \card{f \Poly_\tau \intersect \ZZ^n}
	\]
	which expresses the number of integer lattice points in the $f$-dilate of the open SR-polytope of $\tau$ as a quasi-polynomial.  By Theorem \ref{thm:num-semigroup-sr-polytope-bijection}, this gives the number of numerical semigroups with Frobenius number $f$ and sporadic addition table $\tau$ .  By summing over $\tau$, we conclude that
	\[
	N(n, f) = \sum_{\tau \in \script{T}_n} (-1)^{d_\tau} q_\tau(-f)
	\]
	is itself a quasi-polynomial in the Frobenius number $f$.
	
	To conclude the desired behavior of the highest order term of this quasi-polynomial, notice that only a single addition table defines a full-dimensional SR-polytope, namely, the addition table $\tau_0$ assigning $\infty$ to each pair $(i, j) \in [n]^2$.  This is because any finite value for $\tau(i, j)$ gives a sporadic relation in the definition of the SR-polytope, which lowers the dimension of the resulting convex body\footnote{The semigroups described by the addition table $\tau_0$ are the so-called \emph{elementary} numerical semigroups (with $n+1$ sporadic elements), defined as the semigroups with Frobenius number at most twice the multiplicity.}.
	
	Thus to understand the behavior of the highest order term of $N(n, f)$ as a quasi-polynomial in $f$, it is sufficient to analyze the behavior of $\card{f \Poly_{\tau_0} \intersect \ZZ^n}$.  The polytope $f \Poly_{\tau_0}$ is described by the inequalities $0 < x_1 < \cdots < x_n < f$ and $x_i + x_j > f$ for $i, j \in [n]$.  The integer solutions to these inequalities are thus the subsets of $(f/2, f) \intersect \ZZ$ of size $n$, which allows $\binom{f - \lfloor f/2 \rfloor - 1}{n}$ possibilities.  By parity, this is
	\[
	\binom{f/2 - 1/2}{n} \text{ for $f$ odd}, \qquad\qquad \binom{f/2 - 1}{n} \text{ for $f$ even}
	\]
	In either case, expanding the binomial coefficient as a polynomial in $f$ gives a highest order term of $f^n / (2^n n!)$, with lower order terms depending on the parity of $f$.  Since this expression is the only contribution of a degree $n$ term to the quasi-polynomial expression for $N(n, f)$, we conclude that the leading term of $N(n, f)$ is $f^n / (2^n n!)$ as claimed.
\end{proof}

\section{Generalization to affine semigroups} \label{sec:affine-setting}

Next we extend the approach of the previous section to a higher dimensional setting.  We begin by defining the class of semigroups which will be the focus of our efforts.

An \defn{affine semigroup} is the set of nonnegative integer linear combinations of a finite collection of integer vectors.  If $\vect{x}_1, \ldots, \vect{x}_k \in \ZZ^d$, then the affine semigroup $S$ generated by these vectors can be written as the image $X \NN^k$ of $\NN^k$ under multiplication by the matrix $X \in \ZZ^{d \times k}$ with columns given by the vectors $\vect{x}_i$.  Associated to an affine semigroup is a \defn{rational polyhedral cone} $\Cone$ given by the nonnegative \emph{real} linear combinations of the vectors $\vect{x}_i$, which can similarly be written as $X \RR_{\geq 0}^k$.  The cone $\Cone$ is called \defn{pointed} if there exists a hyperplane $H \subseteq \RR^d$ such that $\Cone \intersect H = \{0\}$, or equivalently, if the only element of $\Cone$ whose additive inverse is in $\Cone$ is 0.  An affine semigroup $S$ is called pointed if its associated rational polyhedral cone is.  An integer lattice point in $\Cone \setminus S$ is called a \defn{gap} of $S$.  A fact which will be useful later is that any submonoid of $\Cone \intersect \ZZ^d$ with finitely many gaps is an affine semigroup.  This is proven in \cite{gubeladze_polytopes_2009}, Lemma 2.9 (Gordan's lemma) and Corollary 2.10.

Throughout this section, we fix a pointed rational polyhedral cone $\Cone$, and restrict our attention to the collection of pointed affine semigroups with associated cone $\Cone$ and finitely many gaps.  We will adopt the notation of~\cite{GMV18} and refer to such semigroups as $\Cone$-semigroups.

In order to generalize SR-polytopes to the setting of affine semigroups, we need an appropriate substitute for the notion of sporadic elements of a numerical semigroup.  To this end, we use a $d$-dimension rational polytope $P \subseteq \RR^d$ containing a neighborhood of the origin, which acts as a scaling ``window'' in which the semigroup elements with a specified addition table reside.

\begin{definition}
	Let $P \subseteq \RR^d$ be a rational polytope containing a neighborhood of the origin in $\Cone$.  We say that $P$ is \defn{$\Cone$-compatible} if for any $x \in \Cone \setminus P$, the translated cone $x + \Cone$ is a subset of $\Cone \setminus P$.
\end{definition}

In particular, a polytope $P$ being $\Cone$-compatible ensures that the union of an affine semigroup $S \subseteq \Cone$ with the set $\{ 0 \} \union ((\Cone \setminus \alpha P) \intersect \ZZ^d)$ is again an affine semigroup for any $\alpha > 0$, which we will need in our subsequent arguments.

\begin{example}
	Suppose $\Cone'$ is a pointed rational polyhedral cone with $\Cone' \supseteq \Cone$, and let $\vect{v} \in (\Cone \intersect \QQ^d) \setminus \{0\}$.  Then
	\[
	P = (\vect{v} - \Cone') \intersect \Cone
	\]
	is a polytope containing a neighborhood of 0 in $\Cone$.  Notice that if $\vect{x}, \vect{y} \in \Cone$ have $\vect{x} + \vect{y} \in P$, then we can write $\vect{x} + \vect{y} = \vect{v} - \vect{w}$ for some $\vect{w} \in \Cone'$.  Then in particular, $\vect{x} = \vect{v} - (\vect{y} + \vect{w})$ and $\vect{y} = \vect{v} - (\vect{x} + \vect{w})$ are both elements of $P$ as well since $\vect{x} + \vect{w}, \vect{y} + \vect{w} \in \Cone'$.  Thus if $\vect{x} \in \Cone \setminus P$ and $\vect{y} \in \Cone$, then $\vect{x} + \vect{y} \in \Cone \setminus P$.  We conclude that $P$ is $\Cone$-compatible.
	
	A reasonably canonical polytope of this type is obtained by choosing $\Cone' = \Cone$, and choosing $\vect{v}$ to be the sum of the \emph{primitive ray generators} of $\Cone$, defined as the set of vectors of smallest length with integer coordinates on each of the extremal rays of $\Cone$.
\end{example}

\begin{example}
	Another simple $\Cone$-compatible polytope is obtained by truncating $\Cone$ using an appropriate hyperplane.  Let $\vect{w} \in \QQ^d$ such that $\vect{w} \cdot \vect{x} > 0$ for every non-zero $\vect{x} \in \Cone$, and let $H_{\vect{w}} = \set{\vect{x} \in \RR^d}{\vect{w} \cdot \vect{x} \leq 1}$.  Then we can define
	\[
	P = H_{\vect{w}} \intersect \Cone
	\]
	If $\vect{x} \in \Cone \setminus P$ and $\vect{y} \in \Cone$, then $\vect{w} \cdot \vect{x} > 1$ and $\vect{w} \cdot \vect{y} > 0$, and thus $\vect{w} \cdot (\vect{x} + \vect{y}) > 1$.  This implies that $\vect{x} + \vect{y} \in \Cone \setminus P$, and so we have that $P$ is $\Cone$-compatible.
\end{example}

For a fixed pointed rational polyhedral cone $\Cone \subseteq \RR^d$, a $\Cone$-compatible polytope $P$, and a positive integer $n$, we wish to describe the rate of growth (with respect to a scaling factor $\alpha$) of the number of affine semigroups $S$ which satisfy
\begin{itemize}
	\item $S \subseteq \Cone$
	\item $\card{S \intersect \alpha P} = n$
	\item All gaps of $S$ are contained in $\alpha P$
\end{itemize}
We write $\Semi_{\Cone,P}(n,\alpha)$ for the set of affine semigroups satisfying these three properties.  We will show in Theorem \ref{thm:BijectionAllDim} that there is a bijection between the semigroups in $\Semi_{\Cone,P}(n,\alpha)$ and the integer lattice points in the $\alpha$-dilates of a finite collection of rational polytopes in $(\RR^d)^k$ depending on $\Cone$, $P$, and $n$.

In the remainder of this section, we will adopt a modest abuse of notation by allowing ``polytopes'' to have both strict and non-strict bounding inequalities, meaning that some faces are allowed to be open.  This will allow us to define an appropriate class of ``polytopes'' for counting the affine semigroups in $\Semi_{\Cone,P}(n,\alpha)$, and in particular it is compatible with the results we use from Ehrhart theory.

When representing affine semigroups as integer lattice points, we will represent a semigroup by its $n$ elements $\vect{x}_1, \ldots, \vect{x}_n \in \ZZ^d$ which are inside the bounded set $S \intersect \alpha P$.  Each $\vect{x}_i$ has $d$ coordinates, and we will think of $\vect{x}_i$ as a column vector with $d$ entries $x_{i,j}$, $j \in [d]$.  Ultimately, we will describe inequalities in the variables $x_{i,j}$ which express the three properties satisfied by the semigroups in $\Semi_{\Cone,P}(n, \alpha)$.

We will need two additional tools to translate these properties into concrete inequalities.  First, recall (c.f.~\cite[Chapter~1]{ziegler_lectures_1995}) that any rational polyhedron $P$, which includes rational polytopes and rational polyhedral cones, can be written as a finite intersection of half-spaces: $P = \intersect_{i=1}^k H_i$, where $H_i = \set{\vect{x} \in \RR^d}{\lambda_i(\vect{x}) \geq b_i}$ are half-spaces defined by linear functionals $\lambda_i$ with integer coefficients, and integers $b_i$.  In particular, the vectors of coefficients of the linear functionals $\lambda_i$ give the \emph{inward pointing normal vectors} to the facets of $S$.

Second, we need a way to express a total ordering of vectors in $\RR^d$ in terms of linear inequalities.  This will ensure that we are able to interpret the affine semigroup elements inside $S \intersect \alpha P$ as a set without overcounting due to reordering of elements.  Recall that the \emph{lexicographic order} on $\RR^d$ is defined by $\vect{u} <_{\lex} \vect{v}$ if $\vect{u} \neq \vect{v}$ and the first non-zero entry of $\vect{v} - \vect{u}$ is positive.  The relation $<_{\lex}$ is a total ordering on $\RR^d$, and further satisfies the property that if $\vect{w} \in \RR^d$ and $\vect{u} <_{\lex} \vect{v}$ then $\vect{u} + \vect{w} <_{\lex} \vect{v} + \vect{w}$.  More generally, if $A$ is an invertible $d \times d$ matrix, we can put a total order on $\RR^d$ by
\[
\vect{u} <_{A} \vect{v} \iff A\vect{u} <_{\lex} A\vect{v}
\]
In fact, every \emph{monomial order} on $\ZZ^d$ arises in this way from an integer-valued matrix $A$ (see~\cite{Robbiano-1986}).  In the following we will define a family of inequalities based on lexicographic order, but in principle the construction can be carried out using any ordering $<_A$.

We now would like to express with concrete inequalities the following condition on vectors $\vect{x}_1, \ldots, \vect{x}_n \in \RR^d$
\begin{equation}
	\label{eq:ineqs0}
	\vect{x}_1 <_{\lex} \vect{x}_2 <_{\lex} \ldots <_{\lex} \vect{x}_n
\end{equation}
It is not possible to express this condition using a single set of linear relations because the lexicographic ordering involves different cases depending on which coordinate is being compared between two vectors.  Accordingly, the vectors satisfying this condition can instead be expressed as a union of $d^{n-1}$ different rational polyhedral cones, corresponding to the $d$ possible coordinates to compare for each of the $n-1$ lexicographic comparisons.  For instance, there are $d$ different cones in $(\RR^d)^2$ whose disjoint union is the set of pairs $(\vect{x}_1,\vect{x}_2)$ satisfying $\vect{x}_1<_{\lex}\vect{x}_2$:
\begin{align*}
	\vect{x}_{1,1} &< \vect{x}_{2,1}\\
	\vect{x}_{1,1} &= \vect{x}_{2,1}, \text{ and } \vect{x}_{1,2} < \vect{x}_{2,2}\\
	&         \hspace{2cm} \vdots \\
	\vect{x}_{1,1} &= \vect{x}_{2,1}, \text{ and } \vect{x}_{1,2} = \vect{x}_{2,2}, \text{ and } \; \ldots \; \text{ and } \vect{x}_{1,d-1} < \vect{x}_{2,d-1} \\
	\vect{x}_{1,1} &= \vect{x}_{2,1}, \text{ and } \vect{x}_{1,2} = \vect{x}_{2,2}, \text{ and } \; \ldots \; \text{ and }  \vect{x}_{1,d-1} = \vect{x}_{2,d-1}, \text{ and } \vect{x}_{1,d} < \vect{x}_{2,d},
\end{align*}

We now define sets analogous to the SR-polytopes of Section \ref{sec:polytope-decomposition}, which will be used to enumerate the desired classes of $\Cone$-semigroups.

\begin{definition}
	\label{def:relpolytopehighdim}
	Let $\Cone \subseteq \RR^d$ be a pointed rational polyhedral cone, let $P \subseteq \RR^d$ be a $\Cone$-compatible polytope, let $n \in \NN$, and let $\tau$ be a truncated addition table on $n$ elements.  Then we define $\Poly_\tau^{(\Cone,P)} \subseteq \RR^{d \times n}$ to be the set of tuples $(\vect{x}_1, \vect{x_2}, \ldots, \vect{x}_n) \in \RR^{d \times n}$ satisfying:
	\begin{itemize}
		\item\label{orderineq} \textit{Order inequalities}: $\vect{x}_1 <_{\lex} \vect{x}_2 <_{\lex} \cdots <_{\lex} \vect{x}_n$
		\item\label{coneineq} \textit{Cone inequalities}: $\vect{x}_i \in \Cone \setminus \{0\}$ for $i = 1, \ldots, n$
		\item\label{polytopeineq} \textit{Polytope inequalities}: $\vect{x}_i \in P$ for $i = 1, \ldots, n$
		\item \textit{Sporadic relations}: $\vect{x}_i + \vect{x}_j = \vect{x}_k$ for $\tau(i, j) = k < \infty$
		\item\label{indepineq} \textit{Truncation inequalities}: $\vect{x}_i + \vect{x}_j \notin P$ for $\tau(i, j) = \infty$
	\end{itemize}
\end{definition}

Once $\Poly_\tau^{(\Cone,P)} \subseteq \RR^{d \times n}$ has been defined, it is fairly straightforward to give the intended generalization of Theorem~\ref{thm:num-semigroup-sr-polytope-bijection}, now associating integer lattice points with affine semigroups.  If $P$ is a $\Cone$-compatible polytope and $S$ is a $\Cone$-semigroup, then we say that $S$ has associated addition table $\tau$ in $P$ if its nonzero elements $\vect{x}_1 <_{\lex} \cdots <_{\lex} \vect{x}_n$ in $P$ satisfy $\vect{x}_i + \vect{x}_j = \vect{x}_k$ when $\tau(i, j) = k$, and satisfy $\vect{x}_i + \vect{x}_j \in \Cone \setminus P$ when $\tau(i, j) = \infty$.

\begin{theorem}
	\label{thm:BijectionAllDim}
	Let $\Cone \subseteq \RR^d$ be a pointed rational polyhedral cone, let $P \subseteq \RR^d$ be a $\Cone$-compatible polytope, let $n, \alpha \in \NN$, and let $\tau$ be a truncated addition table on $n$ elements.  Then the integer lattice points of $\alpha \Poly_\tau^{(\Cone,P)}$ are in bijection with the set of affine semigroups $S \in \Semi_{\Cone,P}(n, \alpha)$ whose nonzero elements $\vect{x}_1 <_{\lex} \cdots <_{\lex} \vect{x}_n$ in $\alpha P$ have associated addition table $\tau$ in $\alpha P$, via the correspondence
	\[
	(\vect{x}_1, \ldots, \vect{x}_n) \mapsto \{0\} \union \{\vect{x}_1, \ldots, \vect{x}_n\} \union ((\Cone \setminus \alpha P) \intersect \ZZ^d)
	\]
\end{theorem}

\begin{proof}
	Let $\varphi$ denote the correspondence defined above.  The order inequalities for $\Poly_\tau^{(\Cone,P)}$ imply that each point $(\vect{x}_1, \ldots, \vect{x}_n)$ corresponds with a unique set $\{\vect{x}_1, \ldots, \vect{x}_n\}$, and the cone inequalities and polytope inequalities ensure that $\{\vect{x}_1, \ldots, \vect{x}_n\}$ is disjoint from $\{0\} \union ((\Cone \setminus \alpha P) \intersect \ZZ^d)$.  This implies that $\varphi$ is bijective.
	
	If $S \in \Semi_{\Cone,P}(n, \alpha)$ has nonzero elements $\vect{x}_1 <_{\lex} \cdots <_{\lex} \vect{x}_n$ in $\alpha P$ with associated addition table $\tau$ in $\alpha P$, then its preimage under $\varphi$ is $(\vect{x}_1, \ldots, \vect{x}_n)$.  This point satisfies the order inequalities and the polytope inequalities by the definition of the vectors $\vect{x}_i$, and the cone inequalities by the definition of $\Semi_{\Cone,P}(n, \alpha)$.  It likewise satisfies the sporadic relations and the truncation inequalities for $\tau$ by the definition of the associated addition table.
	
	Now if $(\vect{x}_1, \ldots, \vect{x}_n) \in \alpha \Poly_\tau^{(\Cone,P)} \intersect \ZZ^d$, we need to show that its image $S = \{0\} \union \{\vect{x}_1, \ldots, \vect{x}_n\} \union ((\Cone \setminus \alpha P) \intersect \ZZ^d)$ under $\varphi$ is an affine semigroup with the desired properties.  By the cone inequalities of $\alpha \Poly_\tau^{(\Cone,P)}$, $S$ is contained in $\Cone$, and by the polytope inequalities, $(S \intersect \alpha P) \setminus \{0\} = \{\vect{x}_1, \ldots \vect{x}_n\}$ consists of $n$ nonzero elements, indexed here in lex increasing order.  Since $\alpha P$ is bounded, the fact that $(\Cone \setminus \alpha P) \intersect \ZZ^d \subseteq S$ implies that $S$ has finite complement in $\Cone \intersect \ZZ^d$.
	
	To see that $S$ is an affine semigroup, it is now enough to show that it is closed under taking sums.  Let $x, y \in S$.  If $x = 0$ then $x + y = y \in S$.  If $x \in \Cone \setminus \alpha P$, then $x / \alpha \in \Cone\setminus P$ and $y / \alpha \in \Cone$, so by $\Cone$-compatibility of $P$ we have that $(x+y)/\alpha \in \Cone \setminus P$, and thus that $x + y \in \Cone \setminus \alpha P$.  Since $x, y$ are integer lattice points, this implies $x + y \in (\Cone \setminus \alpha P) \intersect \ZZ^d \subseteq S$.  The last remaining case is if $x = x_i$ and $y = x_j$ for some $i, j \in [n]$.  If $\tau(i, j) = \infty$, then $x_i + x_j \in \Cone \setminus \alpha P$ by the truncation inequalities, so again because $x, y$ are integer lattice points, their sum is in $S$.  If $\tau(i, j) = k \in [n]$ for some $k$, then $x_i + x_j = x_k \in S$ by the sporadic relations.  We conclude that $S$ is an affine semigroup.
	
	Finally, the sporadic relations and the truncation inequalities imply that $S$ has associated addition table $\tau$ in $\alpha P$, as required.
\end{proof}

We now make use of this characterization to demonstrate quasi-polynomial growth of the affine semigroups in $\Semi_{\Cone,P}(n, \alpha)$.  The description of the set $\Poly_\tau^{(\Cone,P)}$ is messier than the corresponding set for numerical semigroups, so we will need the following characterization.

\begin{lemma}\label{lem:HighDimPolytopes}
	With notation as in Definition \ref{def:relpolytopehighdim}, the set $\Poly_\tau^{(\Cone,P)}$ is a finite union of rational polytopes in $\RR^{d \times n}$ with mixed open and closed faces.
\end{lemma}

\begin{proof}
	Notice that the set $\Poly_\tau^{(\Cone,P)}$ can be thought of as the intersection of the five sets of points in $\RR^{d \times n}$ satisfying each of the five classes of restrictions from Definition \ref{def:relpolytopehighdim} individually.  Thus, it is enough for each of these five sets to consist of finitely many rational polyhedra (not necessarily bounded) with mixed open and closed faces, and for at least one of the sets to be bounded.
	
	We have already seen that the order inequalities are satisfied by the points in a union of $d^{n-1}$ rational polyhedral cones.  The cone inequalities and polytope inequalities describe sets which are a single polyhedron, and the points satisfying the polytope inequalities are also bounded because $P$ is.  The sporadic relations are satisfied by a linear subspace of $\RR^{d \times n}$, which is polyhedral.
	
	A little more explanation is needed to show that the truncation inequalities describe a set which is a finite union of rational polyhedra.  The polytope $P$ can be described as the collection of points $\vect{x} \in \RR^d$ satisfying the inequality $A\vect{x} \geq b$ for some matrix $A \in \RR^{\ell \times d}$ and some vector $\vect{b} \in \RR^\ell$.  Then in order for a vector $\vect{x}$ to not lie in $P$, it has to violate at least one of these inequalities.  We can organize the complement of $P$ into (disjoint) polyhedral regions according to which inequality is the first that fails:
	\begin{align*}
		(A\vect{x})_1 &<    b_1 \\
		(A\vect{x})_1 &\geq b_1, \text{ and } (A\vect{x})_2 < b_2 \\
		&         \hspace{2cm} \vdots \\
		(A\vect{x})_1 &\geq b_1, \text{ and } (A\vect{x})_2 \geq b_2, \text{ and } \; \ldots \; \text{ and } (A\vect{x})_{\ell-1} < b_{\ell-1} \\
		(A\vect{x})_1 &\geq b_1, \text{ and } (A\vect{x})_2 \geq b_2, \text{ and } \; \ldots \; \text{ and }  (A\vect{x})_{\ell-1} \geq b_{\ell-1}, \text{ and } (A\vect{x})_\ell < b_\ell,
	\end{align*}
	Similarly to the order inequalities, the set of points satisfying the truncation inequalities in $\RR^{d \times n}$ is then given by $\ell^k$ polyhedral regions, $k = \card{\tau^{-1}(\infty)}$, each determined by choosing one of the above sets of inequalities for each point $\vect{x}_i + \vect{x}_j$ with $\tau(i, j) = \infty$.
\end{proof}

We can now apply Ehrhart theory to describe the growth of the semigroups in $\Semi_{\Cone,P}(n, \alpha)$ as a function of $\alpha$.  The classical theorem of Ehrhart given in Proposition \ref{prop:ehrhart-qp} is stated in Section \ref{sec:polytope-decomposition} in terms of closed polytopes and their relative interiors, but it also extends to the counting function of lattice points in rational polytopes with mixed open and closed faces:

\begin{proposition}
	\label{prop:open-closed-ehrhart-qp}
	Let $P \subseteq \RR^n$ be a rational convex polytope of dimension $d$ with mixed open and closed faces.  Then the function
	\[
	L_P(\alpha) \defeq \card{\alpha P \intersect \ZZ^n}
	\]
	is a quasi-polynomial function of degree $d$.
\end{proposition}

This can be seen as a consequence of a suitable application of inclusion-exclusion to the counting functions of omitted faces.  We conclude with the main enumerative result for the affine setting.

\begin{corollary}
	\label{cor:affine-quasipolynomial-growth}
	Let $\Cone \subseteq \RR^d$ be a pointed rational polyhedral cone, and let $P \subseteq \RR^d$ be a $\Cone$-compatible polytope.  Then for fixed $n \in \NN$, the number of $\Cone$-semigroups with $n+1$ elements in the polytope $\alpha P$ is a quasi-polynomial function of $\alpha$ with degree $nd$.
\end{corollary}

\begin{proof}
	For each addition table $\tau$ on $n$ elements, the integer lattice points in the set $\alpha \Poly_\tau^{(\Cone,P)}$ correspond with the $\Cone$-semigroups of the desired type with associated addition table $\tau$ in $\alpha P$ by Theorem \ref{thm:BijectionAllDim}.  By Lemma \ref{lem:HighDimPolytopes}, $\Poly_\tau^{(\Cone,P)}$ consists of a finite union of polytopes with mixed open and closed faces, so the number $f_\tau(\alpha)$ of lattice points contained in $\alpha \Poly_\tau^{(\Cone,P)}$ is given by the finite sum of quasi-polynomials associated with its polyhedral components by Proposition \ref{prop:open-closed-ehrhart-qp}.  The total number of semigroups can thus be represented as the sum $f(\alpha) = \sum_{\tau} f_{\tau}(\alpha)$ over all addition tables $\tau$ on $n$ elements, which is itself a quasi-polynomial.
	
	To show that $f(\alpha)$ has degree $nd$, note that since Proposition \ref{prop:open-closed-ehrhart-qp} is being applied to polytopes embedded in $\RR^{d \times n}$, their maximum dimension is $nd$, so each quasi-polynomial in the sum representing $f$ has maximum degree $nd$.  Thus it is enough to construct a family of $\Cone$-semigroups of the desired type with growth at least a constant times $\alpha^{nd}$.  To this end, let $p \in P \intersect \Cone$ be a point of maximal distance from the origin, and let
	\[
	P' = P \intersect \set{x \in \RR^d}{x \cdot (p/\abs{p}) > \abs{p}/2}
	\]
	This is the set of points in $P$ whose component in the $p$ direction is larger than $\abs{p}/2$.  Then $P' \intersect \Cone$ is a polytope with positive volume $V$, so the number of integer lattice points $g(\alpha)$ in $\alpha (P' \intersect \Cone) \subseteq \Cone \intersect \alpha P$ satisfies $\lim_\alpha g(\alpha) / \alpha^d = V$.  Any subset of $n$ points in $\alpha (\Cone \intersect P')$ corresponds with a $\Cone$-semigroup by adjoining $0$ and $(\Cone \setminus \alpha P) \intersect \ZZ^d$ because the sum of any two such points is outside of $\alpha P$: if $x, y \in \alpha(P' \intersect \Cone)$, then $(x + y) \cdot p / \abs{p} > \alpha \abs{p}$, so the component of $x+y$ in the $p$ direction is larger than $\alpha \abs{p}$.  Since $p$ was chosen to have maximal distance from the origin among points in $P$, this implies that $x+y \notin \alpha P$, and so $x + y \in (\Cone \setminus \alpha P) \intersect \ZZ^d$.
	
	The number of semigroups of this form is then given by the binomial expression $\binom{g(\alpha)}{n}$, which can be written as a polynomial $h$ of degree $n$ in the variable $g(\alpha)$, with leading coefficient $1/n!$.  In particular, it satisfies
	\[
	\lim_\alpha h(\alpha) / \alpha^{nd} = V / n! > 0
	\]
	This provides a lower bound on $f(\alpha)$ of order $\alpha^{nd}$, which implies that the degree of $f$ as a quasi-polynomial is at least $nd$.
\end{proof}

\nocite{rosales_garcia_09}

\end{document}